\documentclass[11pt]{amsart}
\usepackage{amscd}
\usepackage{amssymb}
\usepackage[all]{xy}
\usepackage{tgtermes}
\usepackage[T1]{fontenc}
\usepackage{hyperref}
\date{22 July 2014}
\title[Graded-Injective Modules]{Another Proof of a Theorem of Van den Bergh 
about Graded-Injective Modules}

\author{Amnon Yekutieli}
\address{Department of  Mathematics,
Ben Gurion University, Be'er Sheva 84105, Israel}
\email{amyekut@math.bgu.ac.il}

\newtheorem{thm}[equation]{Theorem}

\newtheorem{prop}[equation]{Proposition}
\newtheorem{lem}[equation]{Lemma}
\theoremstyle{definition}

\newtheorem{rem}[equation]{Remark}
\newtheorem{exa}[equation]{Example}

\newcommand{\xar}{\xrightarrow}
\newcommand{\opn}{\operatorname}
\newcommand{\cat}[1]{\operatorname{\mathsf{#1}}}

\newcommand{\cd}{\,{\cdot}\,}

\newcommand{\mfrak}[1]{\mathfrak{#1}}

\newcommand{\mrm}[1]{\mathrm{#1}}

\newcommand{\m}{\mfrak{m}}

\renewcommand{\k}{\Bbbk}

\newcommand{\Z}{\mathbb{Z}}
\newcommand{\N}{\mathbb{N}}

\newcommand{\ot}{\otimes}

\newcommand{\til}[1]{\tilde{#1}}

\newcommand{\lb}{\linebreak}

\begin{document}

\maketitle

Suppose $A = \bigoplus_{i \in \N} A_i$ is a left noetherian $\N$-graded ring.
The category of left $\Z$-graded $A$-modules is denoted by $\cat{GrMod} A$. 
Recall that the objects of $\cat{GrMod} A$ are the left graded $A$-modules $M = 
\bigoplus_{i \in \Z} M_i$, and the morphisms are the $A$-linear homomorphisms of 
degree $0$. A {\em graded-injective module} is an injective object $I$ in the 
abelian category $\cat{GrMod} A$. 

The following result is stated as a ``pleasant exercise in homological 
algebra'' in \cite{VdB}. Its proof was communicated to us privately by M. Van 
den Bergh many years ago, and we referred to it as ``quite involved'' in 
\cite[Remark 4.9]{YZ1}. 
The purpose of this note is to give a modified proof of this result.

\begin{thm}[Van den Bergh] \label{thm:1}
Assume $A$ is a left noetherian $\N$-graded ring. Let $I$ be a graded-injective 
left $A$-module. Then the injective dimension of $I$ in the ungraded sense is 
at most $1$. 
\end{thm}

Actually we prove a slightly more general result (Theorem \ref{thm:15}), of 
which Theorem \ref{thm:1} is a special case. As far as we recall, this was also 
proved by Van den Bergh. 

Our proof follows the same strategy as the original proof by Van den Bergh, 
namely passing through the Rees ring. However our proof is somewhat more 
conceptual, in that it isolates the precise ``reason'' for the dimension jump 
(see Lemma \ref{lem:3}, and compare it to Example \ref{exa:15} below). The 
downside of our proof is that is uses derived categories. 

The dimension jump can already be seen in the next easy commutative example. 

\begin{exa} \label{exa:15}
Suppose $\k$ is a field and $A = \k[t]$, the polynomial ring in the variable 
$t$, that has degree $1$. The injective $A$-modules (in the ungraded sense, 
i.e.\ in the category $\cat{Mod} A$) are direct sums of indecomposable 
$A$-modules. The indecomposable injectives are the field of fractions 
$I(0) \cong \k(t)$, and 
the torsion modules $I(\m) \cong \k(t) / A_{\m}$, where $\m$ is a maximal 
ideal of $A$. 

In the category $\cat{GrMod} A$ there are two indecomposable injective objects, 
up to Serre twist:
the module $I(\m_0)$, where $\m_0 := (t)$ is the unique graded maximal 
ideal of $A$; and the graded ring of fractions $J := \k[t, t^{-1}]$. The latter 
is not injective in $\cat{Mod} A$; but it has injective dimension $1$. 
\end{exa}

Observe that projective modules do not have a jump in dimension: 

\begin{prop}
Let $A$ be a $\Z$-graded ring. If $P$ is a graded-projective 
left $A$-module, then $P$ is also a projective $A$-module in the ungraded 
sense. 
\end{prop}

\begin{proof}
Recall that a graded-free $A$-module is a graded module 
$Q \cong \bigoplus_{x \in X} A(n_x)$, where $X$ is some indexing set,
$n_x \in \Z$, and $A(n_x)$ is the Serre twist of $A$ by the integer $n_x$; 
namely $A(n_x)$ is free on a generator in degree $-n_x$. 

We can find a surjection $Q \to P$ in $\cat{GrMod} A$ from such a graded-free 
module $Q$. Since $P$ is graded-projective, this surjection splits
in $\cat{GrMod} A$, so we get 
an isomorphism $Q \cong P \oplus P'$. But $Q$ is free as an ungraded 
$A$-module, and hence $P$ is projective as an ungraded $A$-module, 
\end{proof}

\begin{rem}
We do not know if the noetherian condition in Theorem \ref{thm:1} is necessary. 
\end{rem}

For the proof of Theorem \ref{thm:1} (and Theorem \ref{thm:15}) we need some 
technical tools, that we take from \cite{YZ1, YZ2}. 
For the rest of the paper $A$ is a left noetherian $\N$-graded ring. By default 
all $A$-modules here are left $A$-modules. Let $M, N$ be graded $A$-modules. 
A function $\phi : M \to N$ is said to be an $A$-linear homomorphism of degree 
$i$ if for all $m \in M_j$ and $a \in A$ we have $\phi(m) \subset N_{i + j}$ 
and $\phi(a \cd m) = a \cd \phi(m)$. The set of all such homomorphisms is 
denoted by $\opn{Hom}_{A}^{\mrm{gr}}(M, N)_i$. And we let 
\[ \opn{Hom}_{A}^{\mrm{gr}}(M, N) :=
\bigoplus_{i \in \Z} \opn{Hom}_{A}^{\mrm{gr}}(M, N)_i , \]
which is a graded abelian group. 
Thus 
\[  \opn{Hom}_{\cat{GrMod} A}(M, N) = 
\opn{Hom}_{A}^{\mrm{gr}}(M, N)_0 . \]

We can view $A$ as a filtered ring, with filtration 
$F = \{ F_i(A) \}_{i \in \Z}$ defined by 
$F_i(A) := \bigoplus_{j \leq i} A_j$. 
Let $\til{A}$ be the Rees ring associated to this filtration. Thus 
\begin{equation} \label{eqn:26}
\til{A} := \bigoplus_{i \in \N} F_i(A) \cd t^i \subset A[t] =
A \ot_{\Z} \Z[t] ,
\end{equation}
where $t$ is a variable of degree $1$. 
There is an isomorphism of graded rings 
$\til{A} / t \cd \til{A} \cong A$, and an isomorphism of filtered rings
$\til{A} / (t - 1) \cd \til{A} \cong A$. 
According to \cite[Theorem 8.2]{ATV} the ring $\til{A}$ is left noetherian. 

Given $\til{M} \in \cat{GrMod} \til{A}$, let 
$\opn{sp}_0(\til{M}) := \til{M} / t \cd \til{M}$, which is a graded $A$-module. 
This is an additive functor (specialization)
\[ \opn{sp}_0 : \cat{GrMod} \til{A} \to \cat{GrMod} A . \]
Likewise the operation 
$\opn{sp}_1(\til{M}) := \til{M} / (t - 1) \cd \til{M}$
is an additive functor 
\[ \opn{sp}_1 : \cat{GrMod} \til{A} \to \cat{Mod} A . \]
According to \cite[Lemma 6.3(1)]{YZ1} the functor $\opn{sp}_1$ is exact. 

Let $M \in \cat{Mod} A$. Suppose $F = \{ F_i(M) \}_{i \in \Z}$ is a 
filtration on $M$ that's compatible with the filtration of $A$. 
This means that each $F_i(M)$ is an abelian subgroup of $M$,
$\bigcap_i F_{i}(M) = 0$, $\bigcup_i F_i(M) = M$, and 
$F_i(A) \cd F_j(M) \subset F_{i + j}(M)$. 
Then the Rees module 
\[ \opn{rs}^F(M) := \bigoplus_{i \in \Z} F_i(M) \cd t^i \subset M[t, t^{-1}] 
= M \ot_{\Z} \Z[t, t^{-1}] \]
is a graded module over the ring $\til{A} = \opn{rs}^F(A)$.
There is a canonical isomorphism 
$\opn{sp}_0(\opn{rs}^F(M)) \cong \opn{gr}^F(M)$ 
in $\cat{GrMod} A$, and a canonical isomorphism 
$\opn{sp}_1(\opn{rs}^F(M)) \lb \cong M$ in $\cat{Mod} A$.

A graded $A$-module $M = \bigoplus_{i \in \Z} M_i$
has a filtration $F = \{ F_i(M) \}_{i \in \Z}$ defined by 
$F_i(M) := \bigoplus_{j \leq i} M_j$. 
The corresponding Rees module is denoted by 
$\opn{rs}(M) := \opn{rs}^F(M)$. Thus we get a functor 
\[ \opn{rs} : \cat{GrMod} A \to \cat{GrMod} \til{A} . \]
This functor is exact. There is a canonical isomorphism  
$\opn{sp}_0(\opn{rs}^F(M)) \cong M$
in $\cat{GrMod} A$. 

For an abelian category $\cat{M}$ we have the (unbounded) derived category 
$\cat{D}(\cat{M})$. In it there are the full subcategories 
$\cat{D}^+(\cat{M})$ and $\cat{D}^-(\cat{M})$, whose objects are the complexes
with bounded below and bounded above cohomologies, respectively. 
The category of complexes of objects of $\cat{M}$ is $\cat{C}(\cat{M})$.
There is a localization functor 
$\cat{C}(\cat{M}) \to \cat{D}(\cat{M})$, 
which is the identity on objects, and inverts quasi-isomorphisms. 
We shall consider these abelian categories: 
$\cat{GrMod} \til{A}$, $\cat{GrMod} A$ and $\cat{Mod} A$.

The exact functors $\opn{sp}_1$ and $\opn{rs}$ extend to triangulated functors 
\[ \opn{sp}_1 :  \cat{D}^{}_{}(\cat{GrMod} \til{A}) \to 
\cat{D}^{}_{}(\cat{Mod} A)  \]
and 
\[ \opn{rs} : \cat{D}^{}_{}(\cat{GrMod} A) \to 
\cat{D}^{}_{}(\cat{GrMod} \til{A}) . \]
There is a left derived functor 
\[ \opn{L sp}_0 : \cat{D}^{-}_{}(\cat{GrMod} \til{A}) \to 
\cat{D}^{}_{}(\cat{GrMod} A) . \]
We use graded-projective resolutions to construct it. 
The right derived functors 
\[ \opn{RHom}_{A}^{\mrm{gr}}(-,-) : \cat{D}^{-}_{}(\cat{GrMod} A)^{\mrm{op}} 
\times \cat{D}^{+}_{}(\cat{GrMod} A) \to \cat{D}^{+}_{}(\cat{GrMod} \Z) \]
and 
\[ \opn{RHom}_{\til{A}}^{\mrm{gr}}(-,-) : 
\cat{D}^{-}_{}(\cat{GrMod} \til{A})^{\mrm{op}} 
\times \cat{D}^{+}_{}(\cat{GrMod} \til{A}) \to \cat{D}^{+}_{}(\cat{GrMod} \Z) \]
are constructed either by graded-projective resolutions of the first argument, 
or by graded-injective resolutions of the second argument.

The key to the dimension jump is the next lemma, that reduces the question to 
the graded ring $\Z[t]$. Observe that there is an obvious graded ring 
homomorphism $\Z[t] \to \til{A}$, and the element $t$ is not a zero divisor in 
$\til{A}$ (cf.\ formula (\ref{eqn:26}). 

\begin{lem} \label{lem:3}
The functor $\opn{L sp}_0$ has cohomological dimension $\leq 1$.
More precisely, if the cohomology of 
$\til{M} \in \cat{D}^{\mrm{b}}(\cat{GrMod} \til{A})$
is concentrated in degrees $\{ i_0, \ldots, i_1 \}$, then 
the cohomology of $\opn{L sp}_0(\til{M})$
is concentrated in degrees $\{ i_0 - 1, i_0, \ldots, i_1 \}$.
\end{lem}

\begin{proof}
Recall that we constructed $\opn{L sp}_0(\til{M})$ by choosing a resolution 
$\til{P} \to \til{M}$, where $\til{P}$ is a bounded above complex of 
graded-projective $\til{A}$-modules, and letting 
$\opn{Lsp}_0(\til{M}) :=  \opn{sp}_0(\til{P})$. 
Since $t$ is not a zero divisor in $\til{A}$, we obtain a short exact sequence 
\begin{equation} \label{eqn:25}
 0 \to \til{P}(-1) \xar{t \cd}  \til{P} \to \opn{sp}_0(\til{P}) \to 0
\end{equation}
in $\cat{C}_{}(\cat{GrMod} \til{A})$. 

Let $\til{Q} \in \cat{C}^{}_{}(\cat{GrMod} \Z[t])$
be the semi-free complex 
\[ \til{Q} := \bigl( \Z[t](-1) \xar{t \cd } \Z[t] \bigr) \]
concentrated in cohomological degrees $-1$ and $0$.  
Using equation (\ref{eqn:25}) we see that there are quasi-isomorphisms 
\[ \til{Q} \ot_{\Z[t]} \til{M} \leftarrow 
\til{Q} \ot_{\Z[t]} \til{P} 
\to \opn{sp}_0(\til{P}) \]
in $\cat{C}(\cat{GrMod} \Z[t])$. The bounds on the cohomology
of the complex $\til{Q} \ot_{\Z[t]} \til{M}$ are clear. 
\end{proof}

\begin{lem} \label{lem:1}
Take $\til{M} \in \cat{D}^{-}_{\mrm{f}}(\cat{GrMod} \til{A})$ and 
$N \in \cat{D}^{+}_{}(\cat{GrMod} A)$. There is an isomorphism 
\[ \opn{RHom}_{\til{A}}^{\mrm{gr}} \bigl( \til{M}, \opn{rs}(N) \bigr) \cong 
\opn{rs} \bigl( \opn{RHom}_{A}^{\mrm{gr}}( \opn{L sp}_0 (\til{M}), N) 
\bigr) \]
in $\cat{D}^{}_{}(\cat{GrMod} \Z[t])$. It is functorial in $\til{M}$ and $N$. 
\end{lem}

\begin{proof}
Any such $\til{M}$ admits a quasi-isomorphism $\til{P} \to \til{M}$, where 
$\til{P}$ is a bounded above complex of finite rank graded-free 
$\til{A}$-modules. There are isomorphisms 
\begin{equation} \label{eqn:15}
\opn{RHom}_{\til{A}}^{\mrm{gr}} \bigl( \til{M}, \opn{rs}(N) \bigr) \cong 
\opn{Hom}_{\til{A}}^{\mrm{gr}} \bigl( \til{P}, \opn{rs}(N) \bigr)
\end{equation}
and 
\begin{equation} \label{eqn:16}
\opn{rs} \bigl( \opn{RHom}_{A}^{\mrm{gr}}( \opn{L sp}_0 (\til{M}), N) \cong
\opn{rs} \bigl( \opn{Hom}_{A}^{\mrm{gr}}( \opn{sp}_0 (\til{P}), N) 
\end{equation}
in $\cat{D}^{}_{}(\cat{GrMod} \Z[t])$, functorial in $\til{P}$. We use the fact 
that $\opn{sp}_0 (\til{P})$ is a bounded above complex of finite rank 
graded-free $A$-modules.

If $\til{P}$ is a single graded-free $\til{A}$-module of rank $1$, i.e.\ 
$\til{P} \cong \til{A}(i)$ for some integer $i$, 
and if $N$ is a single graded $A$-module, then there are isomorphisms
\[ \opn{Hom}_{\til{A}}^{\mrm{gr}} \bigl( \til{P}, \opn{rs}(N) \bigr) \cong 
\opn{rs}(N)(-i) \cong \opn{rs} (N(-i)) \]
and
\[ \opn{rs} \bigl( \opn{Hom}_{A}^{\mrm{gr}}( \opn{sp}_0 (\til{P}), N) \bigr) 
\cong  
\opn{rs} \bigl( \opn{Hom}_{A}^{\mrm{gr}}( A(i), N) \bigr) \cong 
\opn{rs} (N(-i)) \]
in $\cat{GrMod} \Z[t]$. Since these isomorphisms are functorial in the modules 
$\til{P}$ and $N$, we get an isomorphism
\[ \opn{Hom}_{\til{A}}^{\mrm{gr}} \bigl( \til{P}, \opn{rs}(N) \bigr) \cong 
\opn{rs} \bigl( \opn{Hom}_{A}^{\mrm{gr}}( \opn{sp}_0 (\til{P}), N) \bigr) \]
in $\cat{C}^{}_{}(\cat{GrMod} \Z[t])$ for any bounded above complex of finite 
rank graded-free $\til{A}$-modules $\til{P}$, and any bounded below complex
of graded $A$-modules $N$.
Finally we use the isomorphisms (\ref{eqn:15}) and (\ref{eqn:16}).
\end{proof}

\begin{lem} \label{lem:2}
Let $\til{M} \in \cat{D}^{-}_{\mrm{f}}(\cat{GrMod} \til{A})$
and $\til{N} \in \cat{D}^{+}_{}(\cat{GrMod} \til{A})$.
There is an isomorphism 
\[   \opn{sp}_1 \bigl( \opn{RHom}_{\til{A}}^{\mrm{gr}} 
(\til{M}, \til{N}) \bigr) \cong 
\opn{RHom}_{A}^{} \bigl( \opn{sp}_1(M), \opn{sp}_1(N) \bigr) \]
in $\cat{D}^{}_{}(\cat{Mod} \Z)$. It is functorial in $\til{M}$ and $\til{N}$.
\end{lem}

\begin{proof}
This is \cite[Lemma 6.3(3)]{YZ1}. The statement there talks about complexes
$\til{M}, \til{N}$ in $\cat{D}^{\mrm{b}}_{\mrm{f}}(\cat{GrMod} \til{A})$; but 
the proof (the way-out argument) works also for the slightly weaker assumptions 
that we have here. 
\end{proof}

For a graded $A$-module $N$ we denote by 
$\opn{gr{.}inj{.}dim} (N)$ its injective dimension in $\cat{GrMod} A$. 
Thus $I$ is graded-injective iff $\opn{gr{.}inj{.}dim} (I) = 0$. 
Likewise, for an $A$-module $N$  we denote by $\opn{inj{.}dim}(N)$ its 
injective dimension in $\cat{Mod} A$. 

\begin{thm}[Van den Bergh] \label{thm:15}
Assume $A$ is a left noetherian $\N$-graded ring. For any 
graded $A$-module $N$ we have 
\[ \opn{inj{.}dim}(N) \leq \opn{gr{.}inj{.}dim} (N) + 1 . \]
\end{thm}

\begin{proof}
Take a graded $A$-module $N$, and let $d := \opn{gr{.}inj{.}dim} (N)$. 
We may assume that $d < \infty$. 
It is enough to prove that $\opn{Ext}^q_A(M, N) = 0$ for every 
$q > d + 1$ and every finitely generated $A$-module $M$. 

Let $M$ be a finitely generated $A$-module. Let $F$ be a good filtration on 
$M$, namely a filtration such that the graded module 
$\opn{gr}^F(M)$ is finitely generated over $A$. 
Such filtrations exist by \cite[Lemma 7.6.11]{MR}. 
Then (by the graded Nakayama Lemma) the Rees module 
$\til{M} := \opn{rs}^F(M)$ is finitely generated over $\til{A}$.

Consider the graded $\til{A}$-module $\til{N} := \opn{rs}(N)$. 
Since $M \cong  \opn{sp}_1 (\til{M})$ and 
$N \cong  \opn{sp}_1 (\til{N})$ in $\cat{Mod} A$, Lemma \ref{lem:2} says that 
it is enough to prove that  
\[ \mrm{H}^q (\opn{RHom}_{\til{A}}^{\mrm{gr}} (\til{M}, \til{N})) = 0 \]
for all $q > d + 1$. 
By Lemma \ref{lem:1} this is equivalent to proving that
\[ \mrm{H}^q \bigl( \opn{RHom}_{A}^{\mrm{gr}}( \opn{Lsp}_0 (\til{M}), N) 
\bigr) = 0 \]
for $q > d + 1$. Let us take a quasi-isomorphism $N \to I$ 
in $\cat{GrMod} A$, such that $I$ is a complex of graded-injective modules 
concentrated in degrees $\{ 0, \ldots, d \}$. 
We have to prove that 
\[ \mrm{H}^q \bigl( \opn{Hom}_{A}^{\mrm{gr}}( \opn{Lsp}_0 (\til{M}), I) 
\bigr) = 0 \]
for $q > d + 1$.
For that is suffices to show that 
$\mrm{H}^p(\opn{L sp}_0 (\til{M})) = 0$ for $p < -1$. 
This is true by Lemma \ref{lem:3}.
\end{proof}

\medskip \noindent
{\bf Acknowledgments}.  
I wish to thank Michel Van den Bergh for telling me (a long time ago) about his 
result and its original proof. Thanks also to Rishi Vyas for reading an earlier 
version of this note and correcting several mistakes.


\begin{thebibliography}{SGA~4}

\bibitem[ATV]{ATV} M. Artin, J. Tate and M. Van den Bergh, 
Some algebras associated to automorphisms of elliptic curves, 
in ``The Grothendieck Festschrift'', Vol.\ I, pp.\ 33-85, Birkh\"auser,
Boston, 1990.

\bibitem[MR]{MR} J. C. McConnell and J. C. Robson, 
``Noncommutative Noetherian Rings,'' Wiley, 1987.

\bibitem[VdB]{VdB} M. Van den Bergh, 
Existence theorems for dualizing complexes over non-commutative graded and
filtered rings, J. Algebra {\bf 195} (1997), 662-679.

\bibitem[YZ1]{YZ1} A. Yekutieli and J.J. Zhang,
Rings with Auslander Dualizing Complexes,
Journal of Algebra {\bf 213} (1999), 1-51.

\bibitem[YZ2]{YZ2} A. Yekutieli and J.J. Zhang,
Dualizing complexes and perverse modules over differential algebras,
Compositio Math.\ {\bf 141} (2005), 620-654.

\end{thebibliography}
\end{document}